\documentclass[12pt,reqno]{article}
\usepackage[utf8]{inputenc}
\usepackage{geometry}

\usepackage{tikz}
\usepackage{forest}

\usepackage{amsmath}
\usepackage{amssymb}
\usepackage{amsthm}
\usepackage{amsfonts}
\usepackage{mathtools}
\usepackage{commath} 
\usepackage{stackrel}

\usepackage{graphicx}
\usepackage{float}

\usepackage{xcolor}
\usepackage{verbatim}
\usepackage{enumitem}
\usepackage{hyperref}

\theoremstyle{plain}
\newtheorem{theorem}{Theorem}
\newtheorem{corollary}[theorem]{Corollary}

\newtheorem*{lemma*}{Lemma}

\theoremstyle{definition}
\newtheorem{definition}{Definition}
\newtheorem{example}[definition]{Example}

\theoremstyle{remark}

\usepackage[
backend=biber,
style=alphabetic,
sorting=ynt
]{biblatex}
\title{Generalized Eulerian Numbers}

\author{David Dong}

\begin{document}
\maketitle
\begin{abstract}
Let $A(n,m)$ denote the Eulerian numbers, which count the number of permutations on $[n]$ with exactly $m$ descents. It is well known that $A(n,m)$ also counts the number of permutations on $[n]$ with exactly $m$ excedances. In this report, we define numbers of the form $A(n,m,k)$, which count the number of permutations on $[n]$ with exactly $m$ descents and the last element $k$. We then show bijections between this definition and various other analogs for $r$-excedances and $r$-descents. We also prove a variation of Worpitzky's identity on $A(n,m,k)$ using a combinatorial argument mentioned in a paper by Spivey in 2021.
\end{abstract}
\section{Introduction}
Let $A(n,m)$ denote the famous Eulerian numbers, i.e., the number of permutations on $[n]$ with exactly $m$ descents. It is well known that $A(n,m)$ is equal to the number of permutations on $[n]$ with exactly $m$ excedances, through a Foata transform \cite{Stanley}. Eulerian numbers satisfy many other interesting properties. For instance, Worpitzky's identity \cite{Graham} states that 
\[
x^n = \sum_{m = 0}^{n-1} A(n,m) \binom{x+m}{n},
\]
and this statement has a combinatorial interpretation due to \cite{Spivey}. A different combinatorial interpretation, using barred words, is due to \cite{Gessel}. Furthermore, Eulerian numbers satisfy the recurrence relation
\[
A(n,m) = (n-m) A(n-1,m-1) + (m+1) A(n-1,m),
\]
which is a Graham-Knuth-Patashnik recurrence that has been extensively studied \cite{Neuwirth}.

We can naturally generalize these numbers in two different ways. First, we can add a new variable $k$ that considers the last element of permutations: i.e., we can define $A(n,m,k)$ to be equal to the number of permutations $\sigma$ on $[n]$ with exactly $m$ descents, that also satisfy that $\sigma(n) = k$. These numbers have also been studied before: in \cite{Conger}, a closed form for $A(n,m,k)$ is described, and several recurrences are given. 

We can also add another variable to count the size of the descents. In particular, define an $r$-descent of a permutation $\sigma$ to be any location $i \leq n$ satisfying $\sigma(i) \geq \sigma(i+1) + r$. We can also define an $r$-excedance to be any location $i \leq n$ satisfying $\sigma(i) \geq i + r$. Finally, in literature, a $2$-descent is sometimes called a big descent and a $2$-excedance is sometimes called a big excedance. We can then fully generalize Eulerian numbers $A(n,m)$ to numbers $A_r(n,m,k)$, equal to the number of permutations of $\sigma$ on $[n]n$ with $m$ total $r$-descents, satisfying $\sigma(n) = k$. 

A recursive tree is defined to be a tree labeled with values from $0$ to $n-1$, rooted at $0$ such that all paths starting from $0$ and ending at a leaf are increasing. In \cite{Wang}, a bijection between recursive trees and permutations that are counted in generalized Eulerian numbers is proven, as well as relations between descents and special descents in generalized Eulerian numbers. They also raised several conjectures about the relation between the number of permutations counting descents, big descents, excedances, and big excedances.

In this paper, we give background on descents, excedances, Eulerian numbers, and recursive trees in Section~\ref{sec:Prelim}. Then, in Sections~\ref{sec:Bijections} and \ref{sec:Bijections2}, we use our tools to prove theorems on $r$-descents and $r$-excedances. Finally, in Section~\ref{sec:Worpitzky}, we give a combinatorial proof to an extension of Worpitzky's identity involving the numbers $A(n,m,k)$.
\section{Preliminaries}\label{sec:Prelim}
In this section, we define the terms that we use regarding permutations. 
\subsection{Descents and Excedances}
Let $[n]$ denote the range of integers 1 through $n$, and $S_n$ the permutation group on $[n]$. Consider a permutation $\sigma$ as a function $\sigma: [n] \rightarrow [n]$.

An \textit{ascent} is any position $i < n$ where $\sigma(i) < \sigma(i+1)$. 
A very natural extension of this definition is to define a $r$-\textit{ascent} to be any position $i < n$ such that $\sigma(i) + r \leq \sigma(i+1)$. In particular, a 1-ascent is equivalent to an ascent.

The opposite of an ascent, a \textit{descent} is a position $i < n$ with $\sigma(i) > \sigma(i+1)$. In the same fashion, a $r$-\textit{descent} is any position $i < n$ such that $\sigma(i) \geq \sigma(i+1) + r$, with a 1-descent again being equivalent to a descent. 

We define another pair of similar terms: an \textit{excedance} is any position $i \leq n$ satisfying $\sigma(i) > i$, with a $r$-\textit{excedance} being the corresponding generalization, any position $i \leq n$ such that $\sigma(i) \geq i + r$. 

 The opposite, an \textit{anti-excedance} is any position $i \leq n$ satisfying $\sigma(i) < i$, while a $r$-\textit{anti-excedance} is any position $i \leq n$ such that $\sigma(i) \leq i - r$. 

In literature, 2-descents are also often called \textit{big descents}. That is, a big descent is a position $i < n$ with $\sigma(i) - \sigma(i+1) \geq 2$. Similarly, the terms \textit{big ascent}, \textit{big excedance}, and \textit{big anti-excedance} are defined as 2-ascent, 2-excedance, and 2-anti-excedance, respectively.

\subsection{Foata Transform}

We often refer to the Foata transform \cite{Foata}, a bijection from the permutation group $S_n$ to itself. In a Foata transform, take an arbitrary permutation $\sigma \in S_n$. Cut the permutation into blocks right before all elements that are larger than all previous ones, i.e., right before every integer location $i$ where $1 \leq i \leq n$ and $\sigma(i) > \max_{1 \leq j \leq i-1} ( \sigma(j))$. 

Say that the elements in the $i$th block are $a_{i,1}$, $a_{i,2}$, $\ldots$, $a_{i,j}$. Then, take a new permutation $\phi$ such that $\phi(a_{i,k}) = \phi(a_{i,k+1})$ for all valid $i$ and $1 \leq k \leq j$, with $a_{i, j + 1} = a_{i,1}$. In other words, $\phi$ interprets each block as a cycle.

\begin{example}
Consider the following permutation:
\[
\sigma = 
\begin{pmatrix}
1 & 2 & 3 & 4 & 5 & 6 & 7 & 8 \\
5 & 1 & 2 & 8 & 3 & 6 & 4 & 7 \\
\end{pmatrix}.
\]
After cutting the permutation into blocks, we arrive at:
\[
512 \mid 83647.
\]
Interpreting each block as a cycle, we arrive at the following permutation:
\[
\phi = 
\begin{pmatrix}
1 & 2 & 3 & 4 & 5 & 6 & 7 & 8 \\
2 & 5 & 6 & 7 & 1 & 4 & 8 & 3 \\
\end{pmatrix}.
\]
\end{example}

Since the Foata transform is a bijection from $S_n$ to $S_n$, we can also perform an inverse Foata transform with the following steps: we write a permutation $\phi$ in cycle notation, sort the cycles so that the first element in each cycle is the largest element, and then remove the separators between the cycles to form our new permutation, $\sigma$.

The Foata transformation is special because it sends a permutation on $[n]$, denoted $\sigma$, to another permutation $\phi$, then there is a descent at position $i$ in $\sigma$ if and only if there is a corresponding anti-excedance in $\phi$, as $\phi(\sigma(i)) = \sigma(i+1)$ whenever $\sigma(i) > \sigma(i+1)$. Since the Foata transform is a bijection, it can be used to prove that the number of permutations on $[n]$ with $m$ excedances is equal to the number of permutations on $[n]$ with $m$ descents.

\subsection{Eulerian Numbers}
The well-studied Eulerian number $A(n,m)$ is the number of permutations of $[n]$ in which exactly $m$ elements are greater than the previous element. In other words, $A(n,m)$ is the number of permutations in $S_n$ with $m$ ascents. By symmetry, $A(n,m)$ is the number of permutations in $S_n$ with $m$ descents. The following theorem in \cite{Foata} relates Eulerian numbers to excedances as well:
\begin{theorem}
The Eulerian numbers $A(n,m)$ are \textit{also} equal to the number of permutations on $[n]$ with exactly $m$ excedances. By symmetry, $A(n,m)$ also counts the number of permutations on $[n]$ with exactly $m$ anti-excedances.
\end{theorem}

Values of $A(n, m)$ for $0 \leq n \leq 9$ are arranged in the following table \cite{OEIS}: 
\begin{table}[ht!]\label{table:eulerian} 
\centering
\begin{tabular}{c|cccccc}  
    $n,m$ & 0 &	1 &	2 &	3 &	4 &	5 	\\
    1 &	1 &	&	&	&	&		\\		
    2 &	1 &	1 &		&	&		\\		
    3 &	1 &	4 &	1 	&	&		\\		
    4 &	1 &	11 &	11 &	1 	&		\\		
    5 &	1 &	26 &	66 &	26 &	1 &		\\		
    6 &	1 &	57 &	302 &	302 &	57 &	1 	
\end{tabular}
\caption{A table of some small Eulerian numbers.}
\end{table}
The above triangular array, called the Euler triangle, has the recurrence
\[A(n,m) = (n-m) A(n-1,m-1) + (m+1) A(n-1,m).\]  
Note that sum of row $n$ is the factorial $n!$, as it is equal to the number of permutations on $[n]$.

There are several ways to write Eulerian numbers algebraically. For instance, the following closed form exists for Eulerian numbers \cite{Bona}
\[
A(n,m) = \sum_{k=0}^{m+1} (-1)^k \binom{n+1}{k} (m+1-k)^n.
\]
Furhtermore, Worpitzky's identity states that 
\[
x^n = \sum_{m = 0}^{n-1} A(n,m) \binom{x+m}{n}.
\]
Plugging in the above closed form for $A(n,m)$ and then doing algebraic manipulation suffices to prove this identity. However, a combinatorial proof is due to Spivey \cite{Spivey}. We will later be using a technique similar to one used by Spivey to prove an analog for generalized Eulerian numbers.

The Eulerian numbers are the coefficients of the Eulerian polynomials \cite{Graham}:
\[A_{{n}}(x)=\sum _{{m=0}}^{{n}}A(n,m)\ x^{{m}}.\]
The Eulerian polynomials are defined by the exponential generating function
\[\sum _{n=0}^{\infty }A_{n}(t)\cdot {\frac {x^{n}}{n!}}={\frac {t-1}{t-\mathrm {e} ^{(t-1)x}}}.\]
The above formula can be derived using Worpitzky's identity. There also exists an ordinary generating function for the Eulerian numbers, as seen below~\cite{Foata}.
\begin{theorem}\label{thm:eulerianpoly}
The Eulerian polynomial $A_n(x)$ satisfies the following equality:
\[
\frac{A_n(x)}{(1-x)^{n+1}} = \sum_{m=0}^{\infty} (m+1)^n x^m. 
\]
\end{theorem}

\subsection{Recursive Trees}
A \textit{recursive tree}, also known as an \textit{increasing Cayley tree} as defined in \cite{Stanley}, is a tree on $n$ vertices, with labels $0$ through $n-1$. It is rooted at $0$ and satisfies that all paths starting from $0$ and ending at a leaf are strictly increasing. It is \textit{recursive}, because adding $n$ as a leaf to a recursive tree with vertices $0$ through $n-1$ generates a recursive tree with $n+1$ vertices. As such, the number of distinct recursive trees with $n$ vertices is $(n-1)!$.

Define the \textit{smallest rooted path} of a recursive tree to be the path starting at $0$ that always goes to the smallest child. The vertex labeled 0 is called the \textit{root} of the tree.

As an example, here are all of the recursive trees with exactly $4$ vertices, grouped by the value of the leaf at the end of the smallest rooted path.
\begin{table}[ht!]\label{table:recursivetrees}
\centering
\begin{tabular}{|c|c|c|}
\hline
 1 & 2 & 3 \\
\hline
\scalebox{0.75}{
\begin{forest}
[0[1][2[3]]]
\end{forest}
}
\scalebox{0.75}{
\begin{forest}
[0[1][2][3]]
\end{forest}
}
&
\scalebox{0.75}{
\begin{forest}
[0[1[2][3]]]
\end{forest}
}
\scalebox{0.75}{
\begin{forest}
[0[1[2]][3]]
\end{forest}
}
&
\scalebox{0.75}{
\begin{forest}
[0[1[2[3]]]]
\end{forest}
}
\scalebox{0.75}{
\begin{forest}
[0[1[3]][2]]
\end{forest}
}
 \\ \hline
\end{tabular}
\caption{A table containing recursive trees with exactly $4$ vertices.}
\end{table}

Finally, let $R(n,\ell,x)$ be the number of recursive trees with $n$ vertices, $\ell$ vertices of degree $1$ (possibly including the root itself), and with the smallest rooted path ending in $x$.  

In \cite{Stanley}, a bijection between recursive trees with $n$ vertices and permutations on $[n]$ was described using an algorithm equivalent to a depth-first search with the largest element first. Thus, the following theorem was proved:
\begin{theorem}[Theorem 45 in \cite{Wang}]
The value $R(n,\ell,x)$ is exactly equal to the number of permutations on $[n]$ starting with $2$, ending with $x+1$, and with $\ell-1$ descents.
\end{theorem}
We will later be extending this theorem into other analogs for both $r$-descents and $r$-excedances.

In \cite{Wang} it is shown that the values $R(n,\ell,x)$ follow a recursion 
\[
R(n,\ell,x)=\sum_{i=\max(x,2)}^{N-2} R(n-1,\ell-1,i) + \sum_{i=1}^{\max(x-1,1)} R(n-1,\ell,i).
\]

Furthermore, if we define
\[
T(n,\ell)=\sum_{i=0}^{n-1} R(n,\ell,i),
\]
then $T(n,\ell)$ is the number of recursive trees with $n$ vertices and $\ell$ leaves. It is known that $T(n-1,\ell-2)$ equals twice the value of the $2$-Eulerian number $A^2(n,\ell)$. \footnote{The sequence of 2-Eulerian numbers is OEIS sequence A144696. Twice the 2-Eulerian numbers is OEIS sequence A120434 \cite{OEIS}.}. There also exists a closed formula for $T(n,\ell)$, due to \cite{Conger},
\[
T(n,\ell) = \sum_{j=0}^{\ell-2} (-1)^j\cdot (\ell -1 - j)\cdot \binom{n}{j} \cdot (\ell - j)^n.
\]

\section{Bijections Involving \texorpdfstring{$A_r(n,m,k)$}{A(n,m,k)} for Fixed \texorpdfstring{$r$}{r}}\label{sec:Bijections}

In this section, we expand the bijection between recursive trees and permutations found in \cite{Wang} into bijections involving $r$-excedances and $r$-descents for arbitrary $k$. Recall that $A(n,m,k)$ is the number of permutations $\sigma$ on $[n]$ with $m$ descents satisfying that $\sigma(n) = k$. 
\begin{theorem}\label{thm:k-exc-des}
For arbitrary positive integers $n,m,r$ and $k$, the number of permutations of $[n]$ ending with $k$ and with $m$ $r$-descents is equal to the number of permutations of $[n]$ ending with $k$ and with $m$ $r$-excedances.
\end{theorem}
\begin{proof}
Let $D$ be the set of all permutations of $[n]$, ending with $i$, and with $\ell$ $r$-descents, and let $E$ be the set of all permutations of $[n]$, ending with $i$, and with $\ell$ $r$-excedances. We construct a bijection between the two sets.
Our bijection is as follows:
\begin{itemize}
    \item Take a permutation $\sigma \in E$. 
    \item Let $\sigma'$ be the inverse permutation of $\sigma$; i.e., $\sigma'(\sigma(i)) = i$ for all $1 \leq i \leq n$.
    \item Take a Foata transform that sends $\sigma' \rightarrow \phi$ for some permutation $\phi$ on $[n]$.
\end{itemize}

First, we claim that if $\sigma(n) = i$, then performing the bijection gets us a permutation $\phi$ satisfying $\phi(n) = i$. To prove this, note that $\sigma'(i) = n$. Now, when performing the Foata transform, we know that $i$ is in the same cycle as $n$. However, $n$ is the largest element, so it must be the first element of the last block, and therefore $i$, pointing to $n$, must be the last element of the last block. Thus, the blocks of the Foata transform must look like:
\[
(\ldots \mid \ldots \mid \ldots \mid \ldots \ldots \mid n  \ldots i),
\]
and thus, after removing the division between blocks, we have that $\phi(n) = i$.

Inversely, we claim that if $\phi(n) = i$, then performing the inverse bijection gets us a permutation $\sigma$ satisfying $\sigma(n) = i$. The proof is essentially the above argument written backward. Since $i$ is the last element, it must be part of the last block of the Foata transform, implying that $\sigma'(i) = n$. Thus, $\sigma(n) = i$, as expected.

Finally, we claim that there is a bijection between $r$-excedances in $\sigma$ and $r$-descents in $\phi$. Say that there is a $r$-excedance at location $x$ in $\sigma$; i.e., $\sigma(x) = y$ where $y \geq x + k$. Then, we have that $\sigma'(y) = x$. Thus, when writing out the cycle notation, $y$ and $x$ must be consecutive; i.e., there exists an index $j$ such that $\phi(j) = y$ and $\phi(j+1) = x$, implying a $r$-descent at index $j$.

Similarly, we can go backward. Say there is an index $j$ such that $\phi(j) = y$ and $\phi(j+1) = x$ with $y \geq x + k$. Clearly, $x$ and $y$ must be in the same block of the inverse Foata transform and must be consecutive terms. However, this means that $\sigma'(y) = x$, and therefore that $\sigma(x) = y$, as expected.
\end{proof}
\begin{example}
Below is an example with $n = 7, \ell = 2, k = 1, $ and $i = 3$. Thus, we need $\sigma$ to be a permutation on $[7]$, with $2$ descents and last digit $3$. Let us take the example permutation 
    \[
    \sigma =
\begin{pmatrix}
1 & 2 & 3 & 4 & 5 & 6 & 7 \\
6 & 2 & 1 & 4 & 5 & 7 & 3 \\
\end{pmatrix}.
\]
 Now, look at the inverse of the permutation, 
    \[
    \sigma' =
\begin{pmatrix}
1 & 2 & 3 & 4 & 5 & 6 & 7 \\
3 & 2 & 7 & 4 & 5 & 1 & 6 \\
\end{pmatrix}.
\]
Writing the above permutation in standard cycle notation to allow us to take a Foata transform gives
\[
\sigma' = (2)(4)(5)(7613).
\]
Now remove the brackets to create a new permutation, $\phi$. 
\[
\begin{pmatrix}
1 & 2 & 3 & 4 & 5 & 6 & 7 \\
2 & 4 & 5 & 7 & 6 & 1 & 3 \\
\end{pmatrix}.
\]
We can see that the last element is $3$ and that there are exactly two descents. Furthermore, there is a descent between $\phi^{-1}(7)$ and $\phi^{-1}(6)$, as $\sigma(6) = 7$, with another descent between $\phi^{-1}(6)$ and $\phi^{-1}(1)$, as $\sigma(1) = 6$ is the corresponding excedance.
\end{example}

Plugging in $r = 1$ into the above theorem gives us the following corollary:

\begin{corollary}
The generalized Eulerian number $A(n,m,k)$ is also the number of permutations of $[n]$ with $m$ exedances ending in $k$.
\end{corollary}

\begin{theorem}\label{thm:samerequality}
    For any arbitrary positive integers $n,m,r$ and $k_1,k_2 \leq r$, we have that $A_r (n,m,k_1) = A_r(n,m,k_2)$.
\end{theorem}
\begin{proof}
We construct a bijection between permutations counted by $A_r(n,m,k_1)$ and $A_r(n,m,k_2)$. Take some permutation $\sigma$ with $\sigma(n) = k_1$, and with exactly $m$ total $r$-descents. Say that $\sigma(q) = k_2$. Then, construct a new permutation $\phi$ satisfying that
\[
\phi(i) = \sigma(q+i)
\]
for integers $1 \leq i \leq n$, where $\phi(n+i) = \phi(i)$ for all $i \geq 1$. In other words, $\phi$ is a rotation of $\sigma$ so that $\phi(n) = k_2$. Notice that since $k_1,k_2 \leq r$, there cannot be a $r$-descent after $k_1$ in $\phi$ or $k_2$ in $\sigma$. Thus, if there is an $r$-descent at location $i$ in $\sigma$ (i.e., $\sigma(i) - \sigma(i+1) \geq r$), then there is a corresponding $r$-descent at location $i+n-q$ in $\phi$, so $\phi \in A_r(n,m,k_2)$. Since our construction is reversible, it is a bijection, and we are done.
\end{proof}

\section{Bijections Involving \texorpdfstring{$A_r(n,m,k)$}{A(n,m,k)} for Varied \texorpdfstring{$r$}{r}}\label{sec:Bijections2}
In this section, we derive other facts about the numbers $A_r(n,m,k)$. We will also denote the number of permutations satisfying some condition below a symbol. For instance, $\stackbin[\sigma(1) = 1]{}{A(n,m,k)}$ would be equal to the number of permutations $\sigma$ on $[n]$ with $m$ descents ending in $k$ satisfying that $\sigma(1) = 1$. 
\begin{theorem}\label{thm:kequalsone}
For all arbitrary positive integers $n,m,r$, we have:
\[
A_{r+1} (n,m,1) = A_r(n,m,n).
\]
\end{theorem}
\begin{proof}
We use Theorem~\ref{thm:k-exc-des} and let $E$ and $E'$ be the set of permutations on $[n]$ with $m$ total $(r+1)$-excedances ending in $1$, and the set of permutations on $[n]$ with $m$ total $r$-descents ending in $n$, respectively. In particular, $\abs{E} = A_{r+1} (n,m,1)$ and $\abs{E'} = A_r(n,m,n)$. We construct a bijection between $E$ and $E'$. Let $\sigma$ be any permutation in $E$. Then, define $\phi$ to be a permutation on $[n]$ satisfying that for all $1 \leq i \leq n$, 
\[
\phi(i) = 
\begin{cases}
n & i = n \\
\sigma(i) - 1 & \text{otherwise} 
\end{cases}
\]
It is clear that $\phi(n) = n$. Furthermore, there is a $(r+1)$-excedance at location $i$ in $\sigma$ if and only if $\phi(i) = \sigma(i) - 1 \geq i + r$, so $\sigma$ and $\phi$ have the same number of excedances. Thus, $\phi \in E'$, and since the algorithm that we applied is a bijection that sends $E$ to $E'$ it must be the case that $\abs{E} = \abs{E'}$.
\end{proof}
Now, we prove a theorem that allows us to decrease $r$ by one when $k > 1$:
\begin{theorem}\label{thm:variedr}
For all integers $n,m,k$ with $k \geq 2$, we have the equality:
\[
\begin{aligned}
A_{r+1} (n,m,k) &= A_r (n,m+1,k-1) \\
&+ (r-1)\left( A_r(n-1, m, k-1) - A_r(n-1, m+1, k-1)\right).
\end{aligned}
\]
\end{theorem}

\begin{lemma*}
If $k < n$, we have that 
\[
\stackbin[\sigma^{-1}(n) > n-r]{}{A_r(n,m,k)} = (r-1)A_r(n-1, m,k).
\]
\end{lemma*}
\begin{proof}
In fact, we prove the slightly stronger statement that 
\[
\stackbin[\sigma^{-1}(n) = i]{}{A_r(n,m,k)} = A_r(n-1,m,k),
\]
for all $n > i > n-r$. Let $S$ be the set of all permutations $\sigma$ on $[n]$ with $m$ total $r$-excedances and last digit $k$, such that $\sigma^{-1}(n) = i$ for a fixed $i > n-r$. (Thus, $\abs{S} = \stackbin[\sigma^{-1}(n) = i]{}{A_r(n,m,k)}$). Then, let $T$ be the set of all permutations on $[n-1]$ with $m$ total $r$-descents and last digit $k$. From a given permutation $\sigma \in S$ we construct a unique new permutation $\phi \in T$:
\[
\phi(j) = 
\begin{cases}
\sigma(j) & j < i \\
\sigma(j+1) & j \geq i. 
\end{cases}
\]
It is clear to see that as the only term skipped is $\sigma(i) = n$, then $\phi$ is a permutation on $[n-1]$. This algorithm is also clearly invertible, so we only need to show that $\phi \in T$. We have that $\phi(n-1) = \sigma(n) = k$. Finally, if there is an excedance at location $j$ in $\sigma$ (i.e., $\sigma(j) \geq j+r$), then we must have $j \leq n-r$. In fact, $j < n - r$ because $j \neq i \Rightarrow \sigma(j) \neq n$. Thus, there must also be a corresponding excedance at location $j$ in $\phi$, so $\phi$ has $m$ total $r$-excedances, which completes the lemma.
\end{proof}
We also have the following lemma:
\begin{lemma*}
For $1 < k < n$, we have that:
\[
A_{r+1}(n,m,k) = \stackbin[\sigma^{-1}(n) > n-r]{}{A_r(n-1,m,k+1)} + \stackbin[\sigma^{-1}(n) \leq n-r]{}{A_r(n-1,m+1,k-1)}.
\]
\end{lemma*}
\begin{proof}
We prove the slightly stronger claim that 
\[
\stackbin[\sigma^{-1}(1) = i]{}{A_{r+1}(n,m,k)} = 
\begin{cases}
\stackbin[\sigma^{-1}(n) = i]{}{A_{r}(n,m,k-1)} & i \leq n - r \\
\stackbin[\sigma^{-1}(n) = i]{}{A_{r}(n,m+1,k-1)} & i > n - r.
\end{cases}
\]
Summing over all values $i$ completes the proof. To prove this claim, we again let $S$ be the set of all permutations $\sigma$ on $[n]$ with $m$ total $r$-excedances satisfying $\sigma^{-1}(1) = i$, and let $T$ be the set of all permutations $\phi$ on $[n]$ with $m$ total $r$-excedances satisfying $\phi^{-1}(n) = i$, and construct a bijection between $S$ and $T$. We use the following bijection: for any $\sigma \in S$, we take
\[
\phi(j) = 
\begin{cases}
n & \sigma(j) = 1 \\
\sigma(j) - 1 & \text{otherwise} 
\end{cases}
\]
for every $1 \leq j \leq n$. It is clear that this is a bijection; now, we just need to show that $\phi \in T$. Clearly, $\phi$ is a permutation on $[n]$ with $\phi(n) = k-1$. Now, the key is that there is an $(r+1)$-descent at location $j$ in $\sigma$ if and only if there is an $r$-descent at location $j$ in $\phi$, except at location $i$. If $i \leq n-r$, then there is a excedance at $\phi(i)$ but not $\sigma(i)$ (as $\phi(i) = n$ and $\sigma(i)$ = n); if $i > n-r$, there is no excedance at either $\phi(i)$ or $\sigma(i)$.
\end{proof}
\begin{proof}[Proof of~\ref{thm:variedr}]
We can combine the above two lemmas and finish the proof with algebra:
\begin{align*}
A_{r+1}(n,m,k) &= \stackbin[\sigma^{-1}(n) \leq n-r]{}{A_r(n-1,m+1,k-1)} + \stackbin[\sigma^{-1}(n) > n-r]{}{A_r(n-1,m,k+1)} \\
&= A_r(n-1,m+1,k-1)  \\
&- \stackbin[\sigma^{-1}(n) > n-r]{}{A_r(n-1,m+1,k+1)}  + \stackbin[\sigma^{-1}(n) > n-r]{}{A_r(n-1,m,k+1)}  
\end{align*}
which is equal to the desired result after plugging in 
\[
\stackbin[\sigma^{-1}(n) > n-r]{}{A_r(n-1,m+1,k+1)} = (r-1)A_r(n-1, m+1, k-1)
\]
and 
\[
\stackbin[\sigma^{-1}(n) > n-r]{}{A_r(n-1,m,k+1)} = (r-1)A_r(n-1, m, k-1)
\]
by our first lemma.
\end{proof} 

Now, by plugging in $r = 1$ into Theorem~\ref{thm:variedr}, we find the following nice bijection between $1$-descents and $2$-descents:
\begin{corollary}\label{cor: 1-2-exc}
If $k < n$, the value $A(n,m,k)$ is equal to $A_2(n+1,m-1,k+1)$.
\end{corollary}

\subsection{Connections to Recursive Tree Numbers}
Here, we develop a few theorems regarding the relation of recursive tree numbers to the numbers $A(n,m,k)$.
\begin{theorem}
The value $R(n,\ell,i)$ is equal to the number of permutations on $[n-1]$ ending with $i - 1$ and with $\ell - 1$ descents. Thus, it is equal to $A(n-1,\ell-1, i-1)$.
\end{theorem}
\begin{proof}
Using the bijection between recursive trees and permutations, in \cite{Wang}, we have that $R(n,\ell, i)$ is equal to the number of permutations on $[n]$ starting with $2$ and ending with $i+1$ and with $\ell-1$ descents. By Theorem 10 of \cite{Conger}, this is equivalent to the number of permutations of $[n]$, starting with $1$, ending with $i$, and with $\ell - 1$ descents. For each of these permutations, we can remove the first element and decrease every other element by $1$, to get all permutations that end with $i - 1$ and contain $\ell - 1$ descents.
\end{proof}
\begin{corollary}
The value $R(n,\ell,x)$ is equal to the number of permutations on $[n-1]$ ending with $i-1$ and with $\ell-1$ excedances.
\end{corollary}
We can now cite Corollary~\ref{cor: 1-2-exc} to get some equivalent statements:
\begin{corollary}
The value $R(n,\ell,x)$ is equal to the number of permutations on $[n-1]$ ending with $i$ and with $\ell-2$ big excedances.
\end{corollary}
\begin{corollary}
The value $R(n,\ell,x)$ is equal to the number of permutations on $[n-1]$ ending with $i$ and with $\ell-2$ big descents.
\end{corollary}

\section{Generalized Eulerian Polynomials}\label{sec:Worpitzky}
In this section, we develop a way to prove the identity
\[
(x+1)^{n-k+1} x^{k-1} = \sum\limits_{i=0}^{n} A(n,i,k) \binom{x+k}{n-1}.
\]
This has been proven before using barred words \cite{Gessel}, but we use a different (albeit similar) combinatorial proof involving counting functions, generalizing an in \cite{Spivey}. Plugging in $k = 1$ gives Worpitzky's identity after using the fact that $A(n,m,1) = A(n-1,m-1)$. 
\begin{theorem}
The following generalization of Worpitzky's identity holds for integers $n \geq k$:
\[
(x+1)^{n-k+1} x^{k-1} = \sum\limits_{i=0}^{n} A(n,i,k) \binom{x+i}{n-1}.
\]
\end{theorem}
\begin{proof}
Assume $x$ is a positive integer, which suffices for the proof as both sides are polynomials. We double-count the number of functions $f(i): [n] \rightarrow \{0,1,2,\ldots,x\}$, satisfying the following:
\begin{itemize}
    \item The value $f(k)$ is equal to $0$.
    \item For all $i < k$, $f(i)$ is not equal to $0$.
\end{itemize}
Clearly, there are exactly $(x+1)^{n-k+1} x^{k-1}$ valid functions. We now double count the functions. For any permutation $\sigma$ on $[n]$ satisfying that $\sigma(n) = k$ with $i$ descents, we associate a set $S(\sigma)$ that contains exactly $\binom{x+i}{n-1}$ of the above functions. In particular, a given function $f$ is in the set $S(\sigma)$ if and only if for all $1 \leq j \leq n-1$ we have
\[
\begin{cases}
f(\sigma(j)) \geq f(\sigma(j+1)) & \text{descent at }j\text{ in }\sigma \\
f(\sigma(j)) > f(\sigma(j+1)) & \text{otherwise.} \\
\end{cases}
\]
First, we show that $\abs{S}$ is equal to $\binom{x+i}{n-1}$. Let $a_j = f(\sigma(j-1)) - f(\sigma(j))$ for all $0 \leq j \leq n-1$, with $f(\sigma(0)) = x$. It is clear that 
\[
 a_1 + a_2 + \cdots + a_{n} = x
\]
and furthermore that for every $1 \leq j \leq n$,
\[
\begin{cases}
a_j \geq 0 & j = 1 \text{ or descent at }j-1 \\
a_j > 0 & \text{otherwise}.
\end{cases}
\]
It is also direct to show that $f(\sigma(j)) = a_1 + \ldots + a_j$ is a unique, valid, function, given some $a_i$ that satisfy the above conditions. However, since there are exactly $i + 1$ values of $1 \leq j \leq n$ that are either a descent or $0$ or $n$, the number of functions is equal to the number of solutions to
\[
a_1 + a_2 + \cdots + a_n = x + i + 1
\]
which is equal to $\binom{x+i}{n-1}$.

We will show that every function $f: [n] \rightarrow \{0, 1, 2, \ldots, x\}$ is counted in exactly one set $S(\sigma)$ for some permutation $\sigma$ on $[n]$. Choose such an arbitrary function $f$. Note that $f \in S(\sigma)$ is only possible if 
\[
f(\sigma(1)) \geq f(\sigma(2)) \geq \cdots \geq f(\sigma(n)).
\]
Now, we inductively construct our permutation $\sigma$. Say that $f$ has a unique maximum at $M$; i.e., $f(j) < f(M)$ for all $1 \leq j \leq n$ and $j \neq M$. It is clear that $\sigma(1) = M$. Say instead that $f$ has maximal set $S_M$ with at least two elements. This set is defined such that for any $a,b \in S_M$ we have $f(a) = f(b)$, and for any $a \in S_M, b \not\in S_M$, $f(a) > f(b)$. Clearly then, we must have 
\[
f(\sigma(1)) = f(\sigma(2)) = \cdots = f(\sigma(\abs{S_M})),
\]
and furthermore that $S_M = \{\sigma(1),\sigma(2),\ldots,\sigma(\abs{S_M})\}$. However, note that $f(\sigma(j)) = f(\sigma(j+1))$ is only possible if there is a descent at $j$ in $\sigma$, or that $\sigma(1) > \sigma(2) > \cdots > \sigma(S_M)$. This clearly fixes all values $\sigma(1), \sigma(2), \ldots, \sigma(M)$.

We can finish this argument by ``removing" all terms in $S_M$ (or the single maximum term $M$) to show that all values in $\sigma$ are fixed. 
\end{proof}

\section{Acknowledgments}
I am grateful to Tanya Khovanova for introducing me to Eulerian Numb to Ira Gessel for consulting on this project. We thank Rich Wang for his interest in this project. I am also grateful to the MIT PRIMES-USA program for creating such a rare and amazing math research opportunity.

\printbibliography
\end{document}